\documentclass[12pt,a4paper]{article}

\usepackage{graphicx,enumitem}
\usepackage{amsmath,amsthm,amssymb,color}
\usepackage[noadjust,sort]{cite}
\usepackage{hyperref}

\usepackage[font=small,labelfont=bf]{caption}
\normalsize

\newtheorem{thm}{Theorem}[section]

\newtheorem{lemma}[thm]{Lemma}

\numberwithin{equation}{section}

\newcommand{\st}{\mathrel{:}}

\newcommand{\abs}[1]{\lvert#1\rvert} 
\newcommand{\Abs}[1]{\bigl\lvert#1\bigr\rvert} 
\newcommand{\Neq}[1]{N_{#1}}
\newcommand{\Nge}[1]{N_{\ge#1}}

\newcommand{\Req}[1]{\mathcal{R}_{#1}}
\newcommand{\Rge}[1]{\mathcal{R}_{\ge#1}}
\newcommand{\RG}{\mathcal{H}}
\newcommand{\avec}{\boldsymbol{a}}
\newcommand{\Xvec}{\boldsymbol{X}}
\newcommand{\uvec}{\boldsymbol{u}}
\newcommand{\vvec}{\boldsymbol{v}}
\newcommand{\xvec}{\boldsymbol{x}}
\newcommand{\Aut}{\operatorname{Aut}}
\newcommand{\zero}{\mathbf{0}}
\newcommand{\one}{\mathbf{1}}

\DeclareMathOperator{\sech}{sech}

\newcommand{\twonorm}[1]{\mathopen{\|}#1\mathclose{\|}_2}

\def\trans{^{\mathsf{T}}}
\def\Trans#1{\trans\mkern-#1mu}
\renewcommand{\dfrac}[2]{\lower0.12ex\hbox{\large$\textstyle\frac{#1}{#2}$}}
\newcommand{\Dfrac}[2]{\raise0.05ex\hbox{\small$\displaystyle\frac{#1}{#2}$}}
\newcommand{\eps}{\varepsilon}

\newcommand{\Reals}{\mathbb{R}}
\newcommand{\E}{\operatorname{\mathbb{E}}}
\newcommand{\Var}{\operatorname{\mathbb{V}ar}}
\renewcommand{\Pr}{\operatorname{\mathbb{P}}}
\newcommand{\seqnum}[1]{\href{https://oeis.org/#1}{\rm \underline{#1}}}

\def\nicebreak{\vskip 0pt plus 50pt\penalty-300\vskip 0pt plus -50pt }

\begin{document}

\title{Ramsey numbers for regular induced subgraphs}
\author{
Paul W. Dyson \\
\small Independent researcher \\[-0.5ex]
\small Sydney, Australia \\[-0.5ex]
\small\tt pwdyson@yahoo.com
\and
Brendan D. McKay\thanks{Supported by Australian Research Council grant DP190100977} \\
\small School of Computing \\[-0.5ex]
\small Australian National University  \\[-0.5ex]
\small Canberra ACT 2601, Australia \\[-0.5ex]
\small\tt brendan.mckay@anu.edu.au
}

\date{\today}

\maketitle

\begin{abstract}
A problem proposed by Erd\H{o}s, Fajtlowicz and Staton asks for the
smallest $n$ for which every graph on $n$ vertices contains a regular induced
subgraph of order at least~$k$.
A variation is to ask for a regular induced subgraph of order exactly~$k$.
In this paper we provide exact values for $k\le 5$ and lower bounds for $k=6$ and $k=7$.
We also improve the general lower bound of Alon,
Krivelevich and Sudakov [SIAM J. Disc.\ Math., 2008].
\end{abstract}



\section{Introduction}
All the graphs in this paper are undirected and simple.
For positive integers $k,n$, let $\Req k(n)$
denote the set of graphs with $n$ vertices and no induced regular subgraph of order~$k$.
Similarly, let $\Rge k(n)$ denote the set of graphs with $n$ vertices and no
induced regular subgraph of order
\textit{at least}~$k$.
Ramsey's theorem implies that for fixed~$k$
both sets are empty for sufficiently large~$n$.
Also, membership of each set is inherited by
induced subgraphs.
Thus we can define two Ramsey-like parameters:
\begin{align*}
 \Neq k &= \min \bigl\{ n\ge 1 \st 
  \Req k(n)=\emptyset\bigr\}, \\
 \Nge k &= \min \bigl\{ n\ge 1 \st
  \Rge k(n)=\emptyset\bigr\}.
\end{align*}

The literature on these functions is very sparse. Erd\H{o}s, Fajtlowicz and Staton~\cite{E92,E93,E95}
defined inverse functions
\begin{align*}
    f(n) &= \max\bigl\{ k \st n \ge \Nge k\bigr\} , \\
    t(n) &= \max\bigl\{ k \st n \ge R(k,k)\bigr\} ,
\end{align*}
where $R(k,k)$ is the diagonal Ramsey number, and asked two questions:\\
Q1.  Does $f(n)/\log n \to\infty$? \\
Q2.  Does $f(n)-t(n) \to\infty$? \\
An affirmative answer to Q1 would give one for Q2, but both these questions
remain open.

Fajtlowicz et al.\ noted that $\Neq 1=\Nge 1=1$,
$\Neq 2=\Nge 2=2$,
$\Neq 3=6$, $\Nge 3=5$,
$\Neq 4=8$ and $\Nge 4=7$~\cite{FMRS}.
They also gave bounds $\Neq 5\ge 19$, $\Nge 5\ge 12$ and $\Neq 6\ge 18$.

It is clear that in general $\Nge k\le\Neq k\le R(k,k)$, and also that
$\Nge k\le \Nge{k+1}$.
Otherwise very little seems to be known, including the answers to 
these questions:\\
Q3. Is $\Neq k\le \Neq{k+1}$ for $k\ge 1$?\\
Q4. Does $\Neq k-\Nge k\to\infty$ as $k\to\infty$?

The lower bound $\Nge k\ge k^{2-o(1)}$ was proved by 
Bollob\'as~\cite{B} and this was improved by Alon,
Krivelevich and Sudakov~\cite{AKS} to
$\Nge k=\Omega(k^2/(\log k)^{3/2})$.
Both these results were non-constructive and proved using
a heterogeneous random graph.
Note that this is quite different from the case
of a homogeneous random graph of order $n$,
which generally has induced regular subgraphs
of order~$\Theta(n^{2/3})$~\cite{KSW}.
We are not aware of any constructive near-quadratic lower
bounds for $\Nge k$.

For $\Neq k$, Fajtlowicz et al.\ noted that the disjoint union
of $p-1$ cliques of order $p-1$ has no induced regular subgraph of order~$p$, provided $p$ is prime.
In Section~\ref{s:explicit} we will strengthen and
generalise this observation with explicit
constructions, though we do not achieve a quadratic
lower bound for all~$k$.

In this paper, we remove the logarithmic factor from the general bound
of Alon, Krivelevich and Sudakov.
We also
extend knowledge of $\Neq k$ and $\Nge k$
for small~$k$.

\begin{thm}\label{t:improved}
  As $k\to\infty$,
  $\Nge k \ge ((2e)^{-1}-o(1))k^2$.
\end{thm}

\begin{thm}\label{t:comp}
We have $\Neq 5=21$ and $\Nge 5=17$.
Moreover, $\Neq 6\ge 28$, $\Nge 6\ge 21$, $\Neq 7\ge 71$ and $\Nge 7\ge 30$.
\end{thm}

A weaker form of Theorem~\ref{t:improved} is proved in
Section~\ref{s:general}, while Theorem~\ref{t:improved}
is proved using a more complex model in
Section~\ref{s:improved}.
Section~\ref{s:explicit} will give explicit
examples of graphs in $\Req k$ for some
special values of~$k$.
Then in Section~\ref{s:comp} we will describe
our computations and give tables
of counts in $\Req k(n)$ and $\Rge k(n)$.

\section{A weaker quadratic bound}\label{s:general}

Since our proof of Theorem~\ref{t:improved} is quite
complicated, in this section we will show that a much
simpler argument gives a quadratic lower bound,
albeit with a smaller constant.

\begin{thm}\label{t:general}
For sufficiently large $k$,
$\Nge k\ge \frac{9}{163}k^2$.
\end{thm}

As mentioned in the introduction, Alon,
Krivelevich and Sudakov proved the
lower bound $\Nge k=\Omega(k^2/\log^{3/2} k)$~\cite{AKS}.
In this section we will improve their bound
to $\Nge k\ge \frac{9}{163}k^2$.

Define $\RG(k,d)$ to be the set of regular graphs of order
$k$ and degree~$d$.
Define $\lambda=d/(k-1)$ and $K=\binom k2$.

\begin{lemma}\label{regularcount}
  Let $d=d(k)$ be an integer function such that 
  $0\le d\le k-1$ and $dk$ is even for all~$k$.
  Then, as $k\to\infty$, uniformly over $d$ we have
  \[
    \Abs{\RG(k,d)} = \Theta(1)
    \bigl( \lambda^\lambda (1-\lambda)^{1-\lambda}\bigr)^K
    \binom{k-1}{d}^k.
  \]
\end{lemma}
\begin{proof}
 This is a combination of theorems for different ranges
 of $d$ proved in \cite{LW,MW,MW2}.
 Except for the extremes $d=0$ and
 $d=k-1$, where $\Abs{\RG(k,d)}=1$,
 the value represented by $\Theta(1)$
 converges to $\sqrt2 e^{1/4}$.
\end{proof}

Let $\alpha\approx\frac15$ and $\eps\in (0,\alpha)$ be
constants that we will optimise later.
Define
\[
    c_2 = \frac{1}{2(1-\alpha)^2}.
\]

\begin{lemma}\label{uvbound}
  For all $u,v\in[\alpha,1-\alpha]$,
  \[
     \log \Dfrac uv \le \frac{u-v}{v} - c_2 (u-v)^2.
  \]
\end{lemma}
\begin{proof}
 Let $f(u,v)=\log (v/u) + (u-v)/v - c_2 (u-v)^2$.
 Then the derivative $f_v(u,v)$ is $(v-u)(1-2c_2v^2)/v^2$, which
 has the same sign as $v-u$. Therefore the minimum occurs
 when $v=u$, in which case $f(u,v)=0$.
\end{proof}

\begin{lemma}\label{cube}
  Let $z_1,\ldots,z_k$ be independent random variables 
  uniform on $[\alpha-\frac12,\frac12-\alpha]$, and let $\bar z$ be
  their mean.
  Then, for any $\beta>0$,
  \[
      \E \biggl(\exp\Bigl(-\beta\sum_{i=1}^k (z_i-\bar z)^2\Bigr)\biggr)
      \le k^{1/2} \biggl(\frac{\pi}{(1-2\alpha)^2\beta}\biggr)^{(k-1)/2}.
  \]
\end{lemma}
\begin{proof}
 Let $T: (z_1,\ldots,z_k)\mapsto(w_1,\ldots,w_k)$ be an
 orthogonal transformation such that
 $w_k=k^{-1/2}\sum_{i=1}^k z_i$.
 An example is Helmert's transformation~\cite[Section 23:14]{BN}.
 Then we find that $\bar z = k^{-1/2}w_k$
 and $\sum_{i=1}^k (z_i-\bar z)^2 = \sum_{i=1}^{k-1} w_i^2$,
 the last sum not depending on~$w_k$.

 The distribution of $(z_1,\ldots,z_k)$ is uniform
 on the cube $Q=[\alpha-\frac12,\frac12-\alpha]^k$.
 Since $T$ has Jacobian 1 on account of being orthogonal,
 we have
 \begin{align*}
    \E \exp\Bigl(-\beta\sum_{i=1}^k (z_i-\bar z)^2\Bigr)
    &= (1-2\alpha)^{-k}\int_Q
       \exp\Bigl(-\beta\sum_{i=1}^k (z_i-\bar z)^2\Bigr)
        \,dz_1\cdots dz_k \\
    &= (1-2\alpha)^{-k}\int_{T(Q)}
        e^{-\beta\sum_{i=1}^{k-1} w_i^2}\,dw_1\cdots dw_k.
 \end{align*}
 The value of $\abs{w_k}$ in $T(Q)$ is at most
 $\frac12 (1-2\alpha) k^{1/2}$,
 so as an upper bound we can cover $T(Q)$ by $\Reals^{k-1}\times
  [-\frac12 (1-2\alpha) k^{1/2},\frac12 (1-2\alpha) k^{1/2}]$
  to obtain
 \begin{align*}
    \E \exp\Bigl(-\beta\sum_{i=1}^k (z_i-\bar z)^2\Bigr)
    &\le (1-2\alpha)^{-(k-1)} k^{1/2} \int_{\Reals^{k-1}} e^{-\beta\sum_{i=1}^{k-1} w_i^2}
         \,dw_1\cdots dw_{k-1} \\
    &= (1-2\alpha)^{-(k-1)} k^{1/2} \biggl(\, \int_{-\infty}^\infty e^{-\beta x^2}\,dx
             \biggr)^{k-1} \\
    &= k^{1/2} \biggl(\frac{\pi}{(1-2\alpha)^2\beta}\biggr)^{(k-1)/2}. \qedhere
 \end{align*}
\end{proof}

\begin{proof}[Proof of Theorem~\ref{t:general}]
 Let $k\ge \sqrt{163 n/9}$ be an integer.

 Let $\avec=(a_1,\ldots,a_n)$ be a random vector whose
 components are independent random variables from the
 uniform distribution on $[\alpha,1-\alpha]$.
 Generate a random graph $G$ with vertices $\{1,\ldots,n\}$.
 The edges of $G$ appear independently, with edge $ij$
 having probability $(a_i+a_j)/2$.

 Let $G[k]$ denote the subgraph of $G$ induced by vertices
 $\{1,\ldots,k\}$. By symmetry, the probability that $G$ has
 an induced regular subgraph of order $k$
 is bounded above by
  \begin{equation}\label{genbound}
   \binom{n}{k} \sum_{d=0}^{k-1}\, \Pr(G[k] \text{~is $d$-regular}).
 \end{equation}

 We proceed by dividing the range of $d$ into two cases.

 \noindent\textbf{Case (a): $\lambda\notin [\alpha-\eps,1-\alpha+\eps]$}

 By symmetry we can take $\lambda<\alpha-\eps$.
 Conditional on $\avec$, the number of edges is the sum of
 $\binom k2$ independent Bernoulli random variables with
 probabilities in $[\alpha,1-\alpha]$.
 The mean number of edges lies in
 $\bigl[\alpha K,(1-\alpha)K\bigr]$,
 whereas to be $d$-regular for $\lambda<\alpha-\eps$
 requires at most $(\alpha-\eps)K$ edges.
 Using a standard tail bound such as that of 
 McDiarmid~\cite{McD}, we find that there is some 
 constant $c_1>0$ such that
 \[
    \Pr(G[k] \text{~is $d$-regular for some
      $\textstyle \lambda<\alpha-\eps$} \mid \avec)
    \le e^{-c_1k^2}.
 \]
 Since the bound is independent of $\avec$, it also
 holds unconditionally.

 \medskip
 \noindent\textbf{Case (b): $\lambda\in [\alpha-\eps,1-\alpha+\eps]$}
 Take a fixed $d$-regular graph $H\in\RG(k,d)$.
 Define $\bar a = \frac 1k \sum_{i=1}^k a_i$ and
 $y_i = a_i-\bar a$ for $1\le i\le k$.
 Conditional on $\avec$, we have
 \begin{align*}
    \Pr(G[k] = H \mid \avec) &= \prod_{ij\in E(H)} \bigl(\dfrac12(a_i+a_j)\bigr)
      \prod_{ij\notin E(H)} \bigl(1-\dfrac12(a_i+a_j)\bigr) \\
      &=  \prod_{ij\in E(H)} \bigl(\bar a + \dfrac12(y_i+y_j)\bigr)
      \prod_{ij\notin E(H)} \bigl(1 - \bar a - \dfrac12(y_i+y_j)\bigr).
 \end{align*}
 Now apply Lemma~\ref{uvbound} to the first product with
 $u=\bar a+(y_i+y_j)/2, v=\bar a$, and to the second product with
 $u=1-\bar a-(y_i+y_j)/2, v=1-\bar a$.
 This gives
 \begin{align*}
    \Pr(G[k] = H \mid \avec) \le \bar a^m (1-\bar a)^{K-m}
    \exp\biggl(&\, \sum_{ij\in E(H)} 
       \Bigl( \Dfrac1{2\bar a}(y_i+y_j)
       - \Dfrac{c_2}{4} (y_i+y_j)^2\Bigr) \\
     &{\quad}  + \sum_{ij\notin E(H)} 
        \Bigl( -\Dfrac1{2(1-\bar a)}(y_i+y_j) - \Dfrac{c_2}{4}(y_i+y_j)^2\Bigr)\biggr),
 \end{align*}
 where $m=kd/2$.
 Since $\sum_{i=1}^k y_i=0$ and $H$ is regular, the linear terms
 vanish and, moreover
 $\sum_{1\le i<j\le k} (y_i+y_j)^2=(k-2)\sum_{i=1}^k y_i^2$.
 Therefore
 \[
    \Pr(G[k] = H \mid \avec) \le \bar a^m (1-\bar a)^{K-m}
    \exp\biggl( -\Dfrac{c_2}{4} (k-2)\sum_{i=1}^k y_i^2\biggr).
 \]
 Let $s^2$ denote $\sum_{i=1}^k y_i^2$.
 Since the bound above is independent of $H$,
 we have
 \[
     \Pr(G[k] \text{~is $d$-regular} \mid \avec) \le \bar a^m (1-\bar a)^{K-m}
       e^{-c_2(k-2)s^2/4} \,\Abs{\RG(k,d)}.
 \]
 The maximum value of $\bar a^m (1-\bar a)^{K-m}$ occurs when
 $\bar a=\lambda=d/(k-1)$ so, by Lemma~\ref{regularcount},
 \[
   \Pr(G[k] \text{~is $d$-regular} \mid\avec)
   = O(1)\, e^{-c_2(k-2)s^2/4}
     \biggl( \binom{k-1}{d}\lambda^d(1-\lambda)^{k-d-1}\biggr)^k.
 \]
 The quantity in the large parentheses is the value of a binomial
 distribution at its mean.  Applying Stirling's approximation,
 \[
     \binom{k-1}{d}\lambda^d(1-\lambda)^{k-d-1}
     = \frac{1+O(1/k)}{\sqrt{2\pi\lambda(1-\lambda)k}}. 
 \]
 For $\lambda\in[\alpha-\eps,1-\alpha+\eps]$, the smallest
 value of $2\pi\lambda(1-\lambda)$ occurs at the ends, so
 \[
   \binom{k-1}{d}\lambda^d(1-\lambda)^{k-d-1} 
   \le (1+O(1/k))c_3k^{-1/2},
   \text{~~where~~}
    c_3 = \frac{1}{\sqrt{2\pi(\alpha-\eps)(1-\alpha+\eps)}}.
 \]
 Since there are fewer than $k$ possible values of $d$,
 we have
 \begin{equation}\label{3rdbound}
    \Pr(G[k] \text{~is regular}\mid\avec)
    = O(k) e^{-c_2(k-2)s^2/4} c_3^k k^{-k/2}.
 \end{equation}
 
Now we can take the expectation over $\avec$ using
Lemma~\ref{cube} with $\beta=\frac14 c_2(k-2)$.
Since $(k-2)^{(k-1)/2}=\Theta(k^{-1/2}) k^{k/2}$ we obtain
from~\eqref{3rdbound} that
\[
   \Pr(G[k] \text{~is regular})
    = O(k^2) \biggl( \frac{4\pi}{(1-2\alpha)^2 c_2}\biggr)^{(k-1)/2} c_3^k k^{-k}.
\]
The contribution $e^{-c_1k^2}$ from Case~(a) is negligible
in comparison.
Inserting the values of $c_2$ and $c_3$ we obtain
\[
       \Pr(G[k] \text{~is regular}) = O(k^2)
        \biggl( \frac{2(1-\alpha)}
        {(1-2\alpha)\sqrt{(\alpha-\eps)(1-\alpha+\eps)}}
        \biggr)^k k^{-k}.
\]
Finally, as in~\eqref{genbound},
multiply by $\binom nk\le (ne/k)^k$ to cover
all $k$-subsets of $V(G)$, and recall that $n\le \frac{9}{163} k^2$.
This gives that the probability that $G$
contains an induced regular subgraph of order $k$ is
at most
\[
  O(k^2) \biggl(\frac{18 e (1-\alpha)}{163(1-2\alpha)\sqrt{(\alpha-\eps)(1-\alpha+\eps)}}\biggr)^k.
\]
Now setting $\alpha=0.191$ and $\eps=0.0001$, we obtain
the bound $O(0.99986^k k^2)$.
Since this is uniform over $k$, we can
sum over $\{k\st n\le \frac{9}{163} k^2 \text{~and~}k\le n\}$
and still obtain~$o(1)$.
This implies that $\Nge k\ge\frac9{163}k^2$, completing the proof.
\end{proof}

\section{Proof of Theorem~\ref{t:improved}}\label{s:improved}

In this section we prove our main theorem.
The approach is essentially the same as in the previous
section, but the model is more complex.

\begin{lemma}[{\cite[Lemma~5.7]{LV}}]\label{l:lctail}
  If $X$ is a real random variable with a log-concave density,
  mean $\mu$ and standard deviation $s$, then for all $t\ge 0$,
  \[
      \Pr(\abs{X-\mu}\ge ts) \le e^{1-t}.
  \]
\end{lemma}

For $0\le d\le k-1$, define $\lambda=d/(k-1)$ and
\[
   q_\lambda = \binom{k-1}{d}\lambda^d (1-\lambda)^{k-1-d},
\]
with the usual convention that $0^0=1$.

\begin{lemma}\label{l:best}
  Fix $\delta\in(0,\frac12)$.
  Then uniformly for $\lambda\in[\delta,1-\delta]$,
  \[
      q_\lambda = \frac{e^{O(1/k)}}{\sqrt{2\pi\lambda(1-\lambda)(k-1)}}.
  \]
\end{lemma}
\begin{proof}
  This is an immediate consequence of Stirling's formula with remainder.
\end{proof}

Write $\sigma(t)=\frac{e^t}{1+e^t}$ and
$g(t)=\log(1+e^t)$.
Note that
\[
  g'(t) = \sigma(t), \quad
  g''(t) = \dfrac14\sech^2\Dfrac t2, \quad
  g'''(t) = -\dfrac14\tanh\Dfrac t2\sech^2\Dfrac t2,
\]
and these derivatives are all bounded on $\Reals$.

\subsection*{The random graph model}

Fix $T>0$.
Let $\Xvec=(X_1,\ldots,X_n)$ be a random vector whose
components are independent and have a logistic
distribution truncated to~$[-T,T]$.
Explicitly, the density of each $X_i$ is
\[
    \begin{cases}
         \displaystyle\frac{\sech^2 x}{2\tanh T},
          & \,\text{if $x\in[-T,T]$}; \\[0.6ex]
         0, & \,\text{otherwise}. 
    \end{cases}
\]

Conditional on $\Xvec$, 
generate a random graph $G$ with independent edges such
that, for distinct vertices $i,j$, the probability of
an edge $ij$ is $p_{ij}=\sigma(X_i+X_j)$.

Let $G[k]$ be the subgraph induced by the first $k$
vertices of $G$. Since the model is invariant under
permutation of the vertices, we have
\begin{equation}\label{union}
   \Pr(\text{$G$ has an induced regular subgraph of order $k$})
   \le \binom nk \Pr(\text{$G[k]$ is regular}).
\end{equation}
As before, define $\lambda=d/(k-1)$.
With the usual conventions $0^0=1$ and
$0\log 0=0$, define
\begin{align*}
    \psi_\lambda(z) &= g(z) - \lambda z +
  \lambda\log \lambda + (1-\lambda)\log(1-\lambda),\\
E_{\lambda,T} &= \E \exp\biggl( - \sum_{1\le i<j\le k} \psi_\lambda(X_i+X_j)\biggr).
\end{align*}
For $0<\lambda<1$, $\psi'_\lambda(z)=\sigma(z)-\lambda$
and $\psi''_\lambda(z)=\frac14\sech^2(z/2)>0$.
Thus $\psi_\lambda(z)$ is uniquely minimised when
$z=\log(\lambda/(1-\lambda))$, where its value is zero.
Consequently, $\psi_\lambda(z)\ge0$ for all~$z$.
By continuity, $\psi_0(z)\ge0$ and $\psi_1(z)\ge0$
as well.

\begin{lemma}\label{l:Pbounds}
  Assume $0\le d\le k-1$.  Then
  \[
     \Pr(\text{$G[k]$ is $d$-regular})
     = O( q_\lambda^k E_{\lambda,T} ).
  \]
\end{lemma}
\begin{proof}
  Let $H$ be a $d$-regular graph of order~$k$. Then
  \begin{align*}
      \Pr(G[k]=H \mid \Xvec=\xvec) &= \prod_{ij\in E(H)} p_{ij} \prod_{ij\notin E(H)}(1-p_{ij}) \\
      &= \exp\biggl(\, \sum_{ij\in E(H)} \,(x_i+x_j)
            - \sum_{1\le i<j\le k} g(x_i+x_j)\biggr) \\
      &= \bigl(\lambda^\lambda(1-\lambda)^{1-\lambda}\bigr)^K
      \exp\biggl(-\sum_{1\le i<j\le k} \psi_\lambda(x_i+x_j)\biggr),
  \end{align*}
  where $K=\binom k2$.
  Since this quantity is independent of $H$ (the reason for
  our definition of $\sigma(t)$), we can multiply by the
  number of regular graphs from Lemma~\ref{regularcount}
  and then take the expectation over $\Xvec$ to obtain
  \[
    \Pr(\text{$G[k]$ is $d$-regular})
    = \Abs{\RG(k,d)} \bigl(\lambda^\lambda(1-\lambda)^{1-\lambda}\bigr)^K E_{\lambda,T}
    = O( q_\lambda^k E_{\lambda,T} ). \qedhere
  \]
\end{proof}

Define
\begin{align*}
  U(x) &= 2\log\cosh x \\
  \varPhi_\lambda(\xvec) &= \sum_{i=1}^k U(x_i)
   + \sum_{1\le i<j\le k} \psi_\lambda(x_i+x_j).
\end{align*}
Then
\begin{equation}\label{Ilt}
  E_{\lambda,T} = (2\tanh T)^{-k} I_{\lambda,T},
  \quad\text{where}\quad I_{\lambda,T}
   = \int_{[-T,T]^k} e^{-\varPhi_\lambda(\xvec)}\,d\xvec.
\end{equation}

\subsection*{Evaluation of the integral}

In this subsection we estimate $I_{\lambda,T}$.
For matrices $A_1,A_2$, write $A_1\succeq A_2$ if
$\uvec\Trans5 A_1\uvec\ge \uvec\Trans5 A_2\uvec$
for all~$\uvec$.

Define
\[
    b_T = \sigma(-2T) = (1+e^{2T})^{-1},
    \quad
    B_T = b_T(1-b_T) = \dfrac14\sech^2 T.
\]

\begin{lemma}\label{l:Hessian}
   For every $\xvec\in [-T,T]^k$,
   \[
      \nabla^2\varPhi_\lambda(\xvec)\succeq
      B_T((k+6)I+J),
   \]
   where $J$ is the $k\times k$ all-ones matrix.
\end{lemma}
\begin{proof}
  We have $U''(x)=2\sech^2 x \ge 2\sech^2 T = 8B_T$.
  Also, if $x_i,x_j\in[-T,T]$,
  \[
      \psi_\lambda''(x_i+x_j) = g''(x_i+x_j)
      = \dfrac14\sech^2((x_i+x_j)/2)
      \ge B_T.
  \]
  For any $\uvec\in\Reals^k$, we thus have
  \[
     \uvec\Trans1\nabla^2\varPhi_\lambda(\xvec)\uvec
     \ge 8B_T\sum_{i=1}^k u_i^2
       + B_T\sum_{1\le i<j\le k} (u_i+u_j)^2.
  \]
  The lemma follows on expanding the last term using
  $\sum_{i<j} (u_i+u_j)^2=(k-2)\sum_i u_i^2
  + \bigl(\sum_i u_i\bigr)^2$.
\end{proof}

Since by Lemma~\ref{l:Hessian} $\varPhi_\lambda$ is
strictly convex, and it is also invariant under permutation
of the coordinates, it has a unique minimiser in $[-T,T]^k$
of the form $a\one$, where $\one$ is the all-ones vector.

\begin{lemma}\label{l:minimiser}
   $\varPhi_\lambda(\xvec)$ is uniquely minimised at
   $a\one $, where
   \[
      a = \max\bigl\{-T,\min\{a_\ast,T\}\bigr\},
      \quad\text{where}\quad
      a_\ast = \dfrac12\log\frac{d+2}{k+1-d}.
   \]
\end{lemma}
\begin{proof}
  Note that $\psi'_\lambda(2x)=\sigma(2x)-\lambda$.
  This means that the derivative of
  $\varPhi_\lambda(x\one)$ with respect to any
  coordinate is $4\sigma(2x)-2+(k-1)(\sigma(2x)-\lambda)$,
  which vanishes when $x=a_\ast$.
  This gives the unconstrained minimum. If it lies outside
  $[-T,T]^k$, by the convexity of $\varPhi_\lambda(x\one)$
  the minimum is at the nearest corner.
\end{proof}

Define
\[
 \rho = \frac{12\log k}{\sqrt{B_T k}}.
\]
Since $T$ is fixed, $b_T>0$ is fixed and $\rho\to0$.
Consequently, for all sufficiently large $k$, every value
of $\lambda$ considered in Lemmas~\ref{l:firstpart} and
\ref{l:secondpart} lies in
$[b_T/2,1-b_T/2]$. Thus Lemma~\ref{l:best} applies
uniformly in both lemmas with $\delta=b_T/2$.

\begin{lemma}\label{l:concentration}
  Let $a\one $ be the minimiser from Lemma~\ref{l:minimiser}
  and let $\mu_{\lambda,T}$ be the probability measure on
  $[-T,T]^k$ whose density is proportional
  to~$e^{-\varPhi_\lambda}$.
  Then, for all $\lambda\in[0,1]$,
  \[
      \E_{\mu_{\lambda,T}} (X_i-a)^2 \le\frac{1}{B_T(k+6)}
      \qquad (1\le i\le k).
  \]
  In addition, we have
  \[
     \mu_{\lambda,T}\bigl(\max\nolimits_i\, \abs{X_i-a}>\dfrac12\rho\bigr)
     = o(1).
  \]
\end{lemma}
\begin{proof}
  Write $\xvec=a\one+\vvec$ and $W(\vvec)=\varPhi_\lambda(a\one+\vvec) - \varPhi_\lambda(a\one)$.
  By Lemma~\ref{l:Hessian} we have $\nabla^2W(\vvec)\succeq B_T(k+6)I$ for all $\vvec\in [-T-a,T-a]^k$.
  Also, since $\zero$ minimizes $W$ on the convex set
  $[-T-a,T-a]^k$, $\nabla W(\zero)\trans \vvec\ge 0$
  for all $\vvec\in[-T-a,T-a]^k$.
  \begin{equation}\label{vDW}
     \vvec\trans \nabla W(\vvec)
     = \vvec\trans \nabla W(\zero) +
     \int_0^1 \vvec\trans \nabla^2 W(t\vvec)\vvec\,dt
     \ge B_T(k+6)\twonorm{\vvec}^2.
  \end{equation}
  Next, for $1\le i\le k$, integrating by parts with respect to $v_i$
  gives
  \[
      \int_{-T-a}^{T-a} \frac{\partial}{\partial v_i} v_ie^{-W(\vvec)}\,dv_i = \bigl[ v_ie^{-W(\vvec)}\bigr]_{v_i=-T-a}^{v_i=T-a} \ge 0.
  \]
  Noting that $\frac{\partial}{\partial v_i} v_ie^{-W(\vvec)}
  =e^{-W(\vvec)}-v_i\partial_iW(\vvec)e^{-W(\vvec)}$,
  integrate with respect to the other $k-1$ coordinates
  to obtain
  \[
     \int_{[-T-a,T-a]^k} v_i\partial_i W(\vvec)e^{-W(\vvec)}\,d\vvec \le
     \int_{[-T-a,T-a]^k} e^{-W(\vvec)}\,d\vvec.
  \]
  Now sum over $i$ to obtain
  \begin{equation}\label{intint}
     \int_{[-T-a,T-a]^k} \vvec\trans \nabla W(\vvec)e^{-W(\vvec)}\,d\vvec
     \le k \int_{[-T-a,T-a]^k} e^{-W(\vvec)}\,d\vvec.
  \end{equation}
  Apply~\eqref{vDW} to bound $\vvec\trans \nabla W(\vvec)$
  on the left side and divide by the normalizing constant
  (the integral on the right side) to obtain
  $\E_{\mu_{\lambda,T}} \twonorm{\vvec}^2 \le\frac{k}{B_T(k+6)}$.
  By symmetry, this is equivalent to the first claim in the lemma.

  To obtain the second claim, we apply Lemma~\ref{l:lctail}.
  Our bound on $\E_{\mu_{\lambda,T}} v_i^2$ implies that
  both $\abs{\E_{\mu_{\lambda,T}} v_i}$ and
  $\sqrt{\Var_{\mu_{\lambda,T}} v_i}$ are bounded by
  $(B_T(k+6))^{-1/2}$.
  Since marginals of log-concave distributions are 
  log-concave (the Dinghas--Leindler--Pr\'ekopa theorem,
  see~\cite[Thm.~5.1]{LV}), Lemma~\ref{l:lctail} gives
  $\Pr_{\mu_{\lambda,T}}(\abs{v_i} >\frac12\rho)\le e^2k^{-6}$
  for large enough~$k$,
  which implies the second claim of the lemma by
  the union bound.
\end{proof}

Define $z(\lambda)=\frac12\log(\lambda/(1-\lambda))$
and note that $\sigma(2z(\lambda))=\lambda$.

\begin{lemma}\label{l:firstpart}
\[
    \Pr(G[k]\text{~is $d$-regular}) \le e^{o(k)}
    \biggl(\frac{2}{\tanh T}\biggr)^k k^{-k}
\]
uniformly for $b_T+2\rho\le\lambda\le 1-b_T-2\rho$.
\end{lemma}
\begin{proof}
  Since $z(b_T)=-T$ and $z'(\lambda)\ge 2$, the bounds
  on $\lambda$ imply that the box $\varOmega_\lambda
  =\{\xvec\st\max_i\abs{x_i-z(\lambda)}\le\rho\}$
  lies within $[-T,T]^k$.
  The unconstrained optimum $a_\ast$ found in Lemma~\ref{l:minimiser} is $z(\lambda+O(k^{-1}))$,
  and since $z'(\lambda)$ is bounded (for fixed~$T$),
  the minimum is unconstrained (i.e., $a=a_\ast$)
  and $\abs{a-z(\lambda)}=O(k^{-1})$.
  Consequently, $\{\xvec\st\max_i\abs{x_i-a}\le\rho/2\}
  \subseteq \varOmega_\lambda$ and, by 
  Lemma~\ref{l:concentration},
  $\mu_{\lambda,T}(\varOmega_\lambda)=1-o(1)$.

  Now we estimate the integral of $e^{-\varPhi_\lambda}$
  over $\varOmega_\lambda$.
  We calculate
  \begin{gather*}
      U(z(\lambda))=-\log\bigl(4\lambda(1-\lambda)\bigr),\quad
      U'(z(\lambda))=4\lambda-2,\quad
      U''(z(\lambda))=8\lambda(1-\lambda), \\
      \psi_\lambda(2z(\lambda))=\psi'_\lambda(2z(\lambda))=0,
      \qquad\psi''_\lambda(2z(\lambda))=\lambda(1-\lambda),\\
      \nabla^2\varPhi_\lambda(z(\lambda)\one)
      =\lambda(1-\lambda)((k+6)I+J).
  \end{gather*}
  Writing $R(\uvec)$ for the third-order Taylor remainder,
  we have
  \[
     \varPhi_\lambda(z(\lambda)\one+\uvec)
     = -k\log\bigl(4\lambda(1-\lambda)\bigr)
      + (4\lambda-2)\one\trans \uvec
        + \dfrac12 \uvec\Trans3 Q_\lambda\uvec + R(\uvec),
  \]
  where $Q_\lambda=\lambda(1-\lambda)((k+6)I+J)$.
  The third derivatives of both $U$ and $\psi_\lambda$
  are uniformly bounded over~$\Reals$.
  Using $\abs{x+y}^3\le 4(\abs{x}^3+\abs{y}^3)$,
  we have $\abs{R(\uvec)}=O(k^2\rho^3)=o(k)$.

  Since we only need an upper bound, we can now expand
  the domain of the integral to $\Reals^k$ while defining
  $R(\uvec)=0$ outside $\varOmega_\lambda$.
  This gives
  \[
     \int_{\varOmega_\lambda} e^{-\varPhi_\lambda(\xvec)}
        \,d\xvec \le
     e^{o(k)}\bigl(4\lambda(1-\lambda)\bigr)^k (2\pi)^{k/2}
      (\det Q_\lambda)^{-1/2}
      \exp\bigl(2(1-2\lambda)^2\one\Trans3 Q_\lambda^{-1}\one\bigr).
  \]
   Since $\mu_{\lambda,T}(\varOmega_\lambda)=1-o(1)$,
   the integral over $[-T,T]^k$ is $1+o(1)$ times the
   integral over $\varOmega_\lambda$. This factor can
   be absorbed into the~$e^{o(k)}$.
  
  The eigenvalues of $(k+6)I+J$ are $2k+6$ once, and
  $k-1$ of $k+6$. This shows that
  $\exp\bigl(2(1-2\lambda)^2\one\Trans3 Q_\lambda^{-1}\one\bigr)=e^{O(1)}$
  and $\det Q_\lambda = 2\lambda^k(1-\lambda)^k(k+3)(k+6)^{k-1}$.
  Substituting these values and applying Lemmas~\ref{l:best}
  and~\ref{l:Pbounds} gives the required result.
\end{proof}

\begin{lemma}\label{l:secondpart}
Uniformly for $\abs{\lambda-b_T}\le 2\rho$
or $\abs{1-\lambda-b_T}\le 2\rho$ we have
\[
    \Pr(G[k]\text{~is $d$-regular}) \le e^{o(k)}
    \biggl(\frac{2}{\tanh T}\biggr)^k k^{-k}.
\]
\end{lemma}
\begin{proof}
  We give the argument for $\abs{\lambda-b_T}\le 2\rho$
  since the other range is the same on replacing $\xvec$
  by $-\xvec$ and $\lambda$ by $1-\lambda$.

  Let $a_\ast,a$ be as in Lemma~\ref{l:minimiser}.
  Since $z(b_T)=-T$, we have $\abs{a_\ast+T}\le
  \abs{a_\ast-z(\lambda)}+\abs{z(\lambda)-z(b_T)}$.
  From the definition of $a_\ast$, we know
  $a_\ast-z(\lambda)=O(k^{-1})$. Moreover,
  the boundedness of $z'(\lambda)$ implies that
  $z(\lambda)-z(b_T)=O(\abs{\lambda-b_T})=O(\rho)$.
  Thus, $\abs{a_\ast + T}=O(\rho)$. Note that this
  is also true for $\abs{a+T}$ if $a_\ast$ lies
  below~$-T$.

  By Lemma~\ref{l:Hessian}, 
  \[
     \varPhi_\lambda(\xvec) \ge \varPhi_\lambda(a\one)
     + \nabla\varPhi_\lambda(a\one)\trans(\xvec-a\one)
     + \dfrac12 B_T(\xvec-a\one)\trans((k+6)I+J)(\xvec-a\one),
  \]
  and note that we can drop the second term on the right as
  it is nonnegative by the optimality of~$a$.

  Since $\psi_\lambda\ge 0$, $\varPhi_\lambda(a\one)\ge kU(a)$.
  Consequently,
\[
  \int_{[-T,T]^k} e^{-\varPhi_\lambda(\xvec)}\,d\xvec
  \le e^{-kU(a)} (2\pi)^{k/2}
  B_T^{-k/2}
  \bigl(2(k+3)(k+6)^{k-1}\bigr)^{-1/2}.
\]
We have $B_T=\frac14\sech^2 T$ and, from $a=-T+O(\rho)$,
$e^{-U(a)}=\sech^2 a = (1+o(1))\sech^2 T$.
Also, from $b_T=\lambda+O(\rho)$ we have
$B_T=(1+o(1))\lambda(1-\lambda)$.
Applying Lemmas~\ref{l:Pbounds} and~\ref{l:best}, now
establishes the lemma.
\end{proof}

\begin{lemma}\label{l:thirdpart}
  Uniformly for $\lambda\le b_T-2\rho$ or
  $1-\lambda\le b_T-2\rho$,
  \[
     \Pr(G[k]\text{~is $d$-regular}) \le e^{-\Omega(k\log^2 k)}.
  \]
\end{lemma}
\begin{proof}
  The two ranges are equivalent, so we consider just the first.
  Conditional on $\Xvec$, the edges of $G[k]$ are independent
  and have probability at least~$b_T$.
  To be $d$-regular, $G[k]$ must have $\lambda K$ edges.
  Therefore, by McDiarmid's inequality~\cite{McD},
  \[
     \Pr(G[k]\text{~is $d$-regular}\mid \Xvec=\xvec)
     \le \exp\bigl(-2K(b_T-\lambda)^2\bigr) \le e^{-8K\rho^2}.
  \]
  Since this bound is independent of $\xvec$, it holds
  unconditionally.
  Now substituting the value of $\rho$ proves the lemma.
\end{proof}

\begin{proof}[Proof of Theorem~\ref{t:improved}]
Lemmas~\ref{l:firstpart}--\ref{l:thirdpart} cover all ranges
of degrees. Thus we have
\[
    \Pr(G[k]\text{~is $d$-regular}) \le
    \biggl(\frac{2(1+\eta_T(k))}{k\tanh T}\biggr)^k,
\]
for every fixed $T$, where $\eta_T(k)\to 0$
as $k\to\infty$ for any fixed~$T$.
Since there are $k$ possible degrees,
$\binom nk\le\bigl(\dfrac{ne}{k}\bigr)^k$,
and $ke^{o(k)}=e^{o(k)}$,
the union bound gives
\begin{equation}\label{gotit}
    \Pr(G\text{~has an induced regular subgraph of order $k$})
    \le 
    \biggl(\frac{2(1+\eta_T(k))en}{k^2\tanh T}\biggr)^k.
\end{equation}
Given fixed $\eps\in(0,\frac12)$, choose $T$ such that
$\tanh T= 1-\eps$ and $k_0$ sufficiently large that
$k_0\ge \log\eps/\log(1-\eps)$
and $(1+\eta_T(\ell))(1-2\eps)\le(1-\eps)^2$ for all $\ell\ge k_0$.
Set 
\[
   n = \biggl\lfloor\frac{(1-2\eps)k_0^2}{2e}\biggr\rfloor.
\]
Then, applying~\eqref{gotit} for $k=\ell$,
we have for $k_0\le\ell\le n$,
\[
    \Pr(G\text{~has an induced regular subgraph of order $\ell$})
    \le 
    (1-\eps)^\ell.
\]
Summing over $k_0\le\ell\le n$, the probability that $G$
has a regular induced subgraph of order at least $k_0$
is less than~1, so some graph exists without such subgraphs.
Because $\eps$ was chosen arbitrarily,
this is equivalent to the statement of the theorem.
\end{proof}

\section{Explicit constructions}\label{s:explicit}

In this section we note some constructions that
show quadratic lower bounds on $\Neq k$ for some
values of~$k$.
Fajtlowicz et al.\ gave the first example, noting
that the disjoint union of $p-1$ cliques of order
$p-1$ has no induced regular subgraph of order~$p$
if $p$ is prime~\cite{FMRS}.
The lexicographic product $C_{2p-1}[K_{(p-1)/2}]$
has the same property with
$(p-1)/2$ additional vertices, but by using
the union of disjoint lexicographic products we
can do even better.

\begin{lemma}\label{lem-p}
  Consider the lexicographic product $C_r[K_s]$ for $r\ge 4$.
  The connected induced regular subgraphs of degree $d$ are:
  \begin{itemize}\itemsep=0pt
      \item[(i)] A clique of order $d+1$, for $0\le d\le 2s-1$.
      \item[(ii)] A subgraph of order $r(d+1)/3$
         if that is
         an integer, for $2\le d\le 3s-1$,
         which meets every $K_s$-block.
  \end{itemize}
\end{lemma}
\begin{proof}
  Let $H$ be a connected induced regular subgraph.
   Let $B_0,\ldots,B_{r-1}$ be the copies of $K_s$ in
   cyclic order (subscripts modulo $r$) and define $x_i=\Abs{V(H)\cap B_i}$ for each $i$.

  Suppose that for some numbering,
  $x_0=0$, $x_1>0$, $x_2>0$ and $x_3>0$.
  Comparing the degrees of the vertices of $H$
  lying in $B_1$ and $B_2$ we have
  $x_1-1+x_2=x_1+x_2-1+x_3$, which implies
  $x_3=0$, a contradiction.

  Therefore, either $H$ lies within $B_i$ or
  $B_i\cup B_{i+1}$ for some $i$, in which case
  it is a clique, or else $x_i>0$ for all~$i$.

  In the latter case, the regularity of $H$
  implies that $x_i+x_{i+1}+x_{i+2}=d+1$ for all~$i$.
  Subtracting from this the equation
  $x_{i+1}+x_{i+2}+x_{i+3}=d+1$
  gives $x_i=x_{i+3}$, so
  $x_0,\ldots,x_{r-1}$ is periodic of period~3.
  If $r$ is a multiple of 3, we can choose
  $x_0,x_1,x_2$ to obtain any degree in $[2,3s-1]$.
  If $r$ is not a multiple of 3, then
  $x_0=x_1=\cdots=x_{r-1}$, which gives all
  degrees $d\in[2,3s-1]$ such that $d+1$ is a
  multiple of~3.
  In both cases, $H$ has $r(d+1)/3$ vertices.
\end{proof}

\begin{thm}\label{thm-p}
  Let $p\ge 5$ be prime. Then the following graph $G_p$
  has no induced regular subgraph of order~$p$.
  \begin{itemize}\itemsep=0pt
      \item[(a)] If $p=12t+1$, then
        $G_p=3t\, C_9[K_{6t}]$, which has $\frac98(p-1)^2$
        vertices.
      \item[(b)] If $p=12t+5$, then
        $G_p=(3t+1) C_9[K_{6t+2}]$, which has
        $\frac98(p-1)^2$ vertices.
      \item[(c)] If $p=12t+7$, then
        $G_p=C_5[K_{6t+3}]\sqcup (3t+1)C_9[K_{6t+3}]$,
        which has $\frac18(p-1)(9p-7)$ vertices.
      \item[(d)] If $p=12t+11$, then
        $G_p=C_4[K_{6t+5}]\sqcup (3t+2)C_9[K_{6t+5}]$,
        which has $\frac18(p-1)(9p-11)$ vertices.
  \end{itemize}
\end{thm}
\begin{proof}
  In all cases, the independence number and clique
  number of $G_p$ are less than $p$, so an induced
  regular subgraph $H$ of order $p$ and degree $d$ has
  $1\le d \le p-2$.

  By Lemma~\ref{lem-p}, the connected induced
  regular subgraphs of degree $d$ of $C_9[K_s]$
  have order divisible by $d+1$, so in cases
  (a) and (b) $d+1$ must be a divisor of $p$,
  which is impossible as $p$ is prime.

  By the same reasoning, in cases (c) and (d),
  an induced regular subgraph
  of $G_p$ must use a non-clique subgraph $J$ of the
  first component (of which only one is possible).
  
  For case (c), Lemma~\ref{lem-p} says that the order of $J$
  is $\frac53(d+1)$, which is an integer if $d=3m-1$ for
  some integer $m$. But the other components of $H$ have
  order divisible by $d+1$, so we need $p=(5+3q)m$ for
  some integer $q$, which has no solutions when
  $p$ is congruent to 1 modulo~3.

  The same argument for case (d) leads to $p=(4+3q)m$,
  which has no solutions when $p$ is congruent to
  2 modulo~3.
\end{proof}

Although we won't prove it, we believe the graphs in
Theorem~\ref{thm-p} are optimal for $p\ge 13$ within
the class of disjoint unions of lexicographic
products of cycles and cliques.
For $p=7$ and $p=11$, there are better solutions:
$3C_5[K_3]$ with 45 vertices for $p=7$ and
$2C_7[K_5]\sqcup C_9[K_5]$ with 115 vertices for
$p=11$.

\begin{thm}\label{thm-qp}
   If $q<p$ are primes, then
   $\Neq{qp} \ge p^2 + 2q^2p - 4qp + 2
   + (p-1)\min\{q-1,p-q\}$.
\end{thm}
\begin{proof}
  Let $t=\min\{q-1,p-q\}$.
  Construct a graph $G$ as follows.
  Take disjoint cliques $B_1,\ldots,B_{q-1}$
  of order $qp-1$, $A_1,\ldots,A_{p-q}$ of order
  $p-1$, $X_1,\ldots,X_t$ of order $p-1$, 
  and $Y_1,\ldots,Y_{qp-p}$ of order $q-1$.
  For $1\le i\le t$, partition $B_i$ into
  $C_i\cup D_i$ where $\abs{C_i}=qp-p$.
  Finally, for $1\le i\le t$ join all of
  $X_i$ to all of $C_i\cup A_i$.

  Suppose an induced regular subgraph $H$ of order $pq$
  is a union of cliques. For each of the
  possibilities $pqK_1$, $qK_p$, $pK_q$ and $K_{pq}$,
  $G$ does not have enough non-adjacent cliques of
  that size.
  Otherwise, $H$ must have a non-clique component which either
  includes vertices from each of $D_i,C_i,X_i$ for some~$i$
  (which is impossible since the vertices in $C_i$
  would have higher degree in $H$ than those in $D_i$),
  or includes vertices from each of $C_i,X_i,A_i$
  for some $i$, which is similarly impossible.
\end{proof}

The construction in Theorem~\ref{thm-qp} does not work
for $q=4$ as the graph has $2pK_2$ as an induced
subgraph. However, we can achieve slightly less.

\begin{thm}\label{thm-4p}
  For prime $p\ge 7$, $\Neq{4p} \ge p^2+11p-1$.
\end{thm}
\begin{proof}
  Construct a graph $G$ as follows.  Take disjoint
  cliques $A$ of order $4p-1$, $B_1,B_2$ of
  order $2p-1$, $C_1,\ldots,C_{p-4}$ of order $p-1$,
  $D_1,\ldots,D_p$ of order $3$,
  $X_1,X_2,X_3$ of order $p-1$
  and $2p$ isolated vertices.
  Join all of $X_1$ to $3p$ vertices of $A$
  and all of $C_1$.
  Join all of $X_2$ to $p$ vertices of $B_1$ and
  all of $C_2$.
  Finally, join all of $X_3$ to $p$ vertices
  of $B_2$ and all of $C_3$.

  As in the previous theorem, all connected
  induced regular subgraphs are cliques.
  (For example, there is no connected
  regular induced subgraph consisting of non-empty parts
  of $A$, $X_1$ and $C_1$.)
  Counting disjoint non-adjacent cliques that could
  form an induced regular subgraph of order $4p$,
  we find that there are at most $4p-1$ of order~1,
  at most $2p-1$ of order~2, at most $p-1$ of order~4,
  at most 3 of order~$p$, at most one of order $2p$
  and none of order~$4p$.
  This completes the proof.
\end{proof}

\section{Computational investigation}\label{s:comp}

In this section we will describe how our computations
were performed, starting with our generation of
$\Req 5(n)$ and $\Rge 5(n)$ for all $n$, and
$\Req k(n)$ and $\Rge k(n)$ for $k\in\{6,7\}$
and $n\le 13$.

The method of isomorph-free generation is the
canonical construction path algorithm of the
second author~\cite{McK}.
Starting with $K_1$, one vertex at a time is
added in a way that guarantees no duplicates.
We will describe it for $\Req k$, but the same
approach applies to $\Rge k$.
We will assume that graphs in $\Req k(n)$ have
vertices $\{1,\ldots,n\}$ and that vertices are
added in numerical order (so in particular $n$
was the last vertex added).

Given a graph $G\in\Req k(n)$,
and a subset $U\subseteq V(G)$,
let $G{:}U$ denote the graph formed from $G$ by
appending a new vertex $n{+}1$ and joining it to~$U$.
We wish to find a list $\mathcal{U}_G$ of all the
subsets $U\subseteq V(G)$ such that $G{:}U\in\Req k(n{+}1)$.
These subsets are characterised by a set $\mathcal{L}_G$
of pairs of subsets $(U_1,U_2)$ of $\{1,\ldots,n\}$,
where $\abs{U_1}=k-1$ and $U_2\subseteq U_1$,
such that,
if vertex $n{+}1$ is joined to all of $U_2$ but none of
$U_1\setminus U_2$,
then $U_1\cup\{n{+}1\}$ induces a regular subgraph
of order $k$ in $G{:}U$.
The cases are:\\
(1) $U_1$ induces an independent set
of $G$ and $U_2=\emptyset$,\\
(2) $U_1$ induces a
clique of $G$ and $U_2=U_1$,\\
(3) the subgraph
induced by $U_1$ has two degrees $d,d+1$,
$U_2$ is the set of those with degree~$d$,
and $\abs{U_2}=d+1$.\\
Then we have
\[
   \mathcal{U}_G = \{ U\subseteq V(G) \st
      U\cap U_1 \ne U_2 \text{~for all~}
      (U_1,U_2)\in \mathcal{L}_G \}.
\]
We could make $\mathcal{U}_G$ by generating all
of $\mathcal{L}_G$ and then testing all
subsets of $\{1,\ldots,n\}$, but for larger $n$ there
is a more efficient way.
The key observation is that, if $G'=G{:}U$ is in
$\Req k(n{+}1)$, then both $G$ and the subgraph 
of~$G'$ induced by $\{1,\ldots,n-1,n+1\}$ are
in~$\Req k(n)$.
This implies that
\[
  \mathcal{U}_G \subseteq \mathcal{U}_{G-n}
  \cup \{ U\cup \{n\} \st U \in \mathcal{U}_{G-n}\},
\]
which is useful because $\mathcal{U}_{G-n}$ is
already known.
Moreover, we don't need all of $\mathcal{L}_G$
but only those pairs that are not in $\mathcal{L}_{G-n}$,
which means those pairs $(U_1,U_2)\in\mathcal{L}_G$
such that $n\in U_1$.
In essence, we only need to check for induced regular
subgraphs that include both $n$ and~$n{+}1$.

We now briefly describe how the canonical construction
path method works.
We require a tool that can compute the automorphism
group and canonical form of a graph, and for this
we used the second author's package~\texttt{nauty}~\cite{MP}.

For each $G\in\Req k(n)$, we find $\mathcal{U}_G$ as
above and compute its orbits under the action of
$\Aut(G)$.
Then for one representative $U$ of each orbit, we
construct $G{:}U$.
This member of $\Req k(n{+}1)$ is then rejected if
vertex $n{+}1$ is not in the same orbit of
$\Aut(G{:}U)$ as the vertex labelled last in a
canonical labelling.
The theory then implies that exactly
one member of each isomorphism class of $\Req k(n{+}1)$
is accepted~\cite{McK}.

Further speed-ups can be added to this process. For
example, we can assume that the last vertex added to
each graph is a vertex of maximum degree.
This reduces the number of elements of $\mathcal{U}_G$
that must be considered.
Validity requires that the last vertex in a canonical
form has maximum degree, but that can be enforced by
computing the canonical form with the vertices of
maximum degree separated from the remainder.
If $n{+}1$ is the output size, graphs in
$\Req k(n{+}1)$ can be accepted without canonisation
if there is only one vertex of maximum degree; for
smaller sizes the canonical form is computed anyway
since the automorphism group is needed for further
extension.

Another opportunity for optimisation is the observation that
$\Req k(n{+}1)$ is closed under graph complement.
This means that
members of $\Req k(n{+}1)$ with more than $\frac12\binom{n+1}2$
edges can be made from their complements rather than by extending
a graph in $\Req k(n)$.
In particular, graphs in $\Req k(n)$ with
more than $\frac12\binom{n+1}2$ edges don't need to be extended
at all. One of the authors used this optimisation to save
time, while the other avoided it for checking purposes: since
a graph and its complement generally have completely different
construction paths, it is a good check that the output is closed
under complement.

In cases where a complete enumeration was infeasible, we
found lower bounds by generating many graphs of the
largest size we could find, and verified that they 
could not be extended further.
Two techniques proved useful.  In the first technique,
given a graph with $n$ vertices, we took all its
subgraphs with $n{-}1$ vertices and extended them
back to $n$ vertices in all possible ways.
The same was done, but less exhaustively due to the
cost, with smaller subgraphs.
The second technique was to add or remove single edges,
move an edge from $uv$ to $uw$, and
perform switchings of
this form: take edges $uv,xy$ such that $ux,vy$
are not edges, then remove $uv,xy$ and add $ux,vy$.
Usually this results in a graph that has an induced
regular subgraph we don't want, but a small fraction
of cases produce a good graph we can add to the collection.

\subsection{Results}\label{s:results}

We now describe the results of our computations,
as summarised in Theorem~\ref{t:comp}.
The exact counts are given in Table~\ref{Reqtab} and
Table~\ref{Rgetab}.
All exact results were replicated by the two authors
using independent programs, except for the partial
replication of $\Req 5$ mentioned below.
Samples of the largest known graphs in each class
are available on the internet~\cite{DM}.

\begin{itemize}
 \item [$\Req 5$] : In this case we computed
   $\Req 5(n)$ for all~$n$, finding a total of
   42,256,311,802,387 graphs, with the largest being 20,038
   graphs on 20 vertices. This proves
   $\Neq 5=21$. As a check of program correctness,
   the same counts up to 15 vertices were also
   found by an independent program.
   Of the 20,038 extremal graphs, 26 are self-complementary.
 \item [$\Rge 5$] : The full set of graphs was
   announced by the second author in 1997 but
   not formally published. Using two independent
   programs, we confirmed the count of 159,379,295 graphs
   in total, with 954 graphs of order 16 being 
   the largest. Thus $\Nge 5=17$.
   An example of an extremal graph is the
   lexicographic product $P_4[P_4]$, where $P_4$
   is the path on 4 vertices.
   Of the 954 extremal graphs, 24 are self-complementary.
 \item [$\Req 6$] : For this case we found 16 graphs
   with 27 vertices and from 173 to 178 edges,
   but we did not prove that there are 
   none larger. 
   Thus $\Neq 6\ge 28$.
   All of the 27-vertex graphs we found are supergraphs
   of the lexicographic product $C_5[C_5]$;
   for example append a new vertex 
   adjacent to all 25 vertices, then append an
   isolated vertex.
 \item [$\Rge 6$] : This is the next case in 
   increasing order of difficulty beyond those we
   solved completely, but the total number
   of graphs, estimated to be about $6\times 10^{15}$,
   exceeds the computing resources we
   can commit.
   We found 49,871,540 graphs with 20 vertices,
   ranging from 86 to 104 edges.
   None of them are self-complementary,
   and none extend to 21 vertices.
   This proves $\Nge 6\ge 21$.
 \item [$\Req 7$] : In an early edition of this article
   we noted 214,646 graphs in $\Req 7(47)$,
   none of them extending to 48 vertices.
   This turned out to be very suboptimal as we can now
   present a graph in $\Req 7(70)$.
   Start with the unique rank 3 strongly regular graph
   with parameters $(64,27,10,12)$, known as
   $\mathit{VO}^-_6(2)$, 
   and let $S$ be its complement~\cite[p. 296]{BV}.
   Let $v,w$ be two non-adjacent vertices in~$S$;
   since $\mathit{VO}^-_6(2)$ is edge-transitive, it doesn't matter 
   which two.
   Now append six vertices $a,b,c,d,e,f$, and edges
   $va, ab, bc, wd, de, ef$.
   In addition, $c$ is adjacent to the neighbours of
   $v$ in $S$, and $f$ is adjacent to the
   neighbours of $w$ in $S$.
   The resulting graph and its complement are
   in $\Req7(70)$, therefore $\Neq7\ge 71$.
   
 \item [$\Rge 7$] : We found 174,775,920 graphs on 29
   vertices, none of them extending to 30 vertices.
   They have from 172 to 234 edges,
   and none of them are self-complementary.
   This proves $\Nge 7\ge 30$.\\
   To judge how close we are to a complete collection,
   we used the following random process. 
   Start with a heterogeneous random graph $G$
   on 28 vertices using
   the model defined in the proof of Theorem~\ref{t:general}.
   (1) If $G$ belongs to $\Rge7(28)$, extend it to 29 vertices
   in every way possible (which might be none) and stop.
   (2) Otherwise,
   remove a random vertex of $G$
   that lies on at least one regular
   induced subgraph of order 7 or more.
   (3) Then, try to extend the resulting 27-vertex graph
   back to a new 28-vertex graph $G$ in one random
   way such that the new vertex lies on no regular
   induced subgraph of order 7 or more. 
   If this is not possible, stop.
   (4) Repeat from step~(1).
   Usually this iteration stops at step~(3) but occasionally
   a graph in $\Rge7(29)$ is produced at step~(1).
   Of more than 6,000 graphs in $\Rge7(29)$ generated
   by this method,
   none of them belonged to the previous collection of
   more than 174 million graphs.
   This leads us to strongly suspect that our collection
   is far from complete and that the existence of a
   graph in $\Rge7(29)$ that extends to $\Rge7(30)$
   remains a good possibility.
\end{itemize}

Finally, we mention a common generalisation of $\Neq k$ and $\Nge k$.
If $S$ is a set of positive integers, let
$N_S$ denote the least positive $n$ such that every
graph with $n$ vertices has a regular induced
subgraph whose order is in~$S$.
We will not explore this concept in depth,
but we will report that $N_{\{4,5\}}=\Nge 4=7$ and
$N_{\{5,6\}}=\Nge5 = 17$.
The number of extremal graphs is 10 and 9,075, respectively.
The 27-vertex graph in $\Req 6$ that we described
above has no regular induced subgraph of order~7 either,
so $N_{\{6,7\}}\ge 28$.
Finally, we have 605,961,088 graphs on 30 vertices with no induced
regular subgraphs of order 7 or 8, showing that
$N_{\{7,8\}}\ge 31$.

\nicebreak
\section{Integer sequences}

\begin{table}[ht]
\centering
\begin{tabular}{c|@{\hspace{0.6em}}cccc}
  $n$ & $\abs{\Req 4(n)}$ & $\abs{\Req 5(n)}$
  & $\abs{\Req 6(n)}$ & $\abs{\Req 7(n)}$ \\
  \hline
  1 & 1 & 1 & 1 & 1 \\
  2 & 2 & 2 & 2 & 2 \\
  3 & 4 & 4 & 4 & 4 \\
  4 & 7 & 11 & 11 & 11 \\
  5 & 12 & 31 & 34 & 34 \\
  6 & 12 & 136 & 148 & 156 \\
  7 & 2 & 792 & 964 & 1038 \\
  8 & & 7185 & 10472 & 12246 \\
  9 & & 94893 & 191776 & 269646 \\
  10 & & 1714430 & 5524670 & 11453460 \\
  11 & & 37216434 & 219302174 & 907948002 \\
  12 & & 854671213 & 10333796899 & 127924347122 \\
  13 & & 18369802688 & 493296884096 & 30302185606487 \\
  14 & & 328662169364 & 19658348081642 & ..\\
  15 & & 4236467418682 & .. \\
  16 & & 29440587191035 & \\
  17 & & 8014569475958 & \\
  18 & & 216388700196 & \\
  19 & & 373319294 & \\
  20 & & 20038 &
\end{tabular}
\caption{Number of graphs in $\Req k(n)$}
\label{Reqtab}
\end{table}

\begin{table}[ht]
\centering
\begin{tabular}{c|@{\hspace{0.6em}}cccc}
  $n$ & $\abs{\Rge 4(n)}$ & $\abs{\Rge 5(n)}$
  & $\abs{\Rge 6(n)}$ & $\abs{\Rge 7(n)}$ \\
  \hline
  1 & 1 & 1 & 1 & 1 \\
  2 & 2 & 2 & 2 & 2 \\
  3 & 4 & 4 & 4 & 4 \\
  4 & 7 & 11 & 11 & 11 \\
  5 & 11 & 31 & 34 & 34 \\
  6 & 10 & 130 & 148 & 156 \\
  7 & & 728 & 960 & 1038 \\
  8 & & 6027 & 10390 & 12226 \\
  9 & & 66308 & 188560 & 268920 \\
  10 & & 818276 & 5317230 & 11361262 \\
  11 & & 8336902 & 202396620 & 885194426 \\
  12 & & 45933753 & 8905369148 & 119298229792 \\
  13 & & 79888458 & 384098286140 & 25716285392622 \\
  14 & & 23814804 & 13129756210164 & .. \\
  15 & & 512906 & .. \\
  16 & & 954
\end{tabular}
\caption{Number of graphs in $\Rge k(n)$}
\label{Rgetab}
\end{table}

\begin{itemize}\itemsep=0pt
\item \seqnum{A394564} Least integer $a(n)$ such that every graph on $a(n)$ vertices has an induced regular
subgraph of order $n$. 
\item \seqnum{A394574} Greatest $a(n)$ such
that every graph on $n$ vertices has an induced
regular subgraph of order $a(n)$.
\item \seqnum{A394563} Least integer $a(n)$ such that every graph on $a(n)$ vertices has an induced regular subgraph of order
at least $n$.
\item \seqnum{A390257} Minimum size of maximum regular induced subgraph of a graph on $n$ vertices.
\item \seqnum{A394573} Number of graphs with $n$ vertices that have no induced regular subgraph of order 4.
\item \seqnum{A394400} Number of graphs with $n$ vertices that have no induced regular subgraph of order 4 or greater.
\item \seqnum{A394539} Number of graphs with $n$ vertices that have no induced regular subgraph of order 5.
\item \seqnum{A390919} Number of graphs with $n$ vertices that have no induced regular subgraph of order 5 or greater.
\item \seqnum{A394462} Number of graphs with $n$ vertices that have no induced regular subgraph of order 6.
\item \seqnum{A392636} Number of graphs with $n$ vertices that have no induced regular subgraph of order 6 or greater.
\item \seqnum{A394930} Number of graphs with $n$ vertices that have no induced regular subgraph of order 7.
\item \seqnum{A394933} Number of graphs with $n$ vertices that have no induced regular subgraph of order 7 or greater.
\end{itemize}

\nicebreak
\section{Acknowledgments}
In preparing this article, the authors made use of
large language models (LLMs),
particularly ChatGPT, Claude and Gemini.
We found these to be very useful in suggesting methods,
proposing constructions, and checking proofs.
However, they also made frequent mistakes.
Nothing in the final version uses LLM wording directly,
and everything has been carefully checked by its
human authors.

The first author thanks those who donated computers used for his calculations: 
Kay Dyson,
Peter Dyson,
Jane Hope,
Joanne Knight,
Matthew Kwan,
Robin Langer,
Wendy Langer,
Brendan McKay,
Andrew Moylan,
Ard Oerlemans,
Rachel Wong and
Guoxing Zhao.

The second author used computing resources of the
Australian National Computational Infrastructure
and the ARDC Nectar Research Cloud.

\nicebreak

\end{document}